\documentclass[11pt]{article}
\usepackage[T1]{fontenc}
\usepackage[utf8]{inputenc}
\usepackage{lmodern}
\usepackage{microtype}
\usepackage{amsmath,amssymb,amsthm,mathtools}
\usepackage[margin=1in]{geometry}
\usepackage[hidelinks]{hyperref}
\usepackage{enumitem}

\hypersetup{
  pdftitle={Single-Event Multinomial Full Kelly via Implicit State Positions},
  pdfsubject={Kelly criterion, multinomial betting, log-optimal portfolio selection}
}

\newtheorem{theorem}{Theorem}
\newtheorem{proposition}[theorem]{Proposition}
\newtheorem{lemma}[theorem]{Lemma}
\newtheorem{remark}[theorem]{Remark}

\title{Single-Event Multinomial Full Kelly via Implicit State Positions}
\author{Christopher D. Long}
\date{}

\begin{document}
\maketitle

\begin{abstract}
For a single event with finitely many mutually exclusive outcomes, the full Kelly problem is to maximize expected log wealth over nonnegative stakes together with an optional cash position. The optimal formula is classical, but the support-selection step is often presented via Lagrange multipliers. This note gives a shorter state-price derivation. A cash fraction $c$ acts as an implicit position in every outcome: in terminal-wealth terms, it is equivalent to a baseline stake $cq_i$ on outcome $i$, where $q_i$ is the state price. On any active support, explicit bets therefore only top up favorable outcomes from this baseline $cq_i$ to the optimal total stake $p_i$. This yields the formula $x_i = (p_i - c q_i)_+$, the threshold rule $p_i/q_i > c$, and, after sorting outcomes by $p_i/q_i$, a one-pass greedy algorithm for support selection. The result is standard in substance, but the implicit-position viewpoint gives a compact proof and a convenient way to remember the solution.
\end{abstract}


\section{Introduction}
The Kelly criterion maximizes expected logarithmic growth and occupies a central place in gambling, portfolio theory, and information theory \cite{Kelly1956,Breiman1961,Thorp1971,MacLeanThorpZiemba2011}. In the horse-race model, or more generally in a finite state market with state-contingent claims, the optimization problem is concave and completely explicit \cite{CoverThomas2006}. For a single multinomial event with mutually exclusive outcomes, explicit treatments also appear in the pari-mutuel literature, notably in Rosner's early analysis \cite{Rosner1975} and in the closed-form solution of Smoczynski and Tomkins \cite{SmoczynskiTomkins2010}. The natural question is therefore not whether the problem can be solved, but how to present the solution so that the support structure and the formula are both immediate.

This note gives a short answer. The key observation is that cash itself is a state-contingent claim. If the bettor keeps cash fraction $c$, then in state $i$ this contributes wealth $c$; equivalently, relative to outcome $i$, it is as if the bettor were already holding an implicit stake $cq_i$. Once that is recognized, the fixed-support problem collapses to the elementary full-investment Kelly rule, and the global problem becomes a monotone threshold search. The resulting algorithm is greedy: sort by the edge ratio $p_i/q_i$ and keep adding outcomes while that ratio exceeds the current cash level.

Nothing here changes the classical theory. The contribution is expository: a short state-price proof in which cash is an all-state claim, so active bets merely top up favorable outcomes from $cq_i$ to $p_i$. This isolates the idea that makes the formulas memorable and keeps the support-selection step nearly algebra-free.

\section{Model}
Consider a single event with outcomes $i=1,\dots,n$. Let
\[
  p_i>0, \qquad \sum_{i=1}^n p_i=1,
\]
be the bettor's subjective probabilities. Outcome $i$ is available at decimal odds $O_i>1$, and we write
\[
  q_i:=\frac{1}{O_i}>0
\]
for the corresponding state price or implied probability. The overround is $\sum_i q_i$, which may exceed $1$.

A strategy consists of nonnegative stakes $x_i\ge 0$ and a cash holding $c\ge 0$ satisfying
\[
  c+\sum_{i=1}^n x_i=1.
\]
If outcome $i$ occurs, terminal wealth is
\[
  W_i(c,x)=c+\frac{x_i}{q_i}.
\]
The single-period full Kelly problem is
\begin{equation}\label{eq:kelly-problem}
  \max_{c\ge 0,\,x_i\ge 0}\; G(c,x):=\sum_{i=1}^n p_i\log\!\left(c+\frac{x_i}{q_i}\right)
  \quad\text{subject to}\quad c+\sum_{i=1}^n x_i=1.
\end{equation}

It is convenient to define the \emph{edge ratios}
\[
  r_i:=\frac{p_i}{q_i}.
\]
These compare the bettor's probabilities to the market-implied probabilities. They will also be the optimal terminal wealth levels on active outcomes.

Although the terminal wealth vector is unique, the pair $(c,x)$ need not be unique in degenerate tie cases. The canonical strategy constructed below is unique whenever no edge ratio equals the final cash cutoff.

\section{A weighted full-investment lemma}
The only optimization fact needed is the following standard weighted AM-GM/Jensen lemma.

\begin{lemma}\label{lem:full-investment}
Let $a_i>0$ with $A:=\sum_{i=1}^m a_i$, and let $S>0$. Then
\[
  \sum_{i=1}^m a_i\log z_i
\]
is uniquely maximized over $z_i>0$ subject to $\sum_{i=1}^m z_i=S$ at
\[
  z_i=\frac{S}{A}a_i \qquad (i=1,\dots,m).
\]
\end{lemma}

\begin{proof}
Write
\[
  \sum_{i=1}^m a_i\log z_i-\sum_{i=1}^m a_i\log a_i
  =A\sum_{i=1}^m \frac{a_i}{A}\log\!\left(\frac{z_i}{a_i}\right).
\]
By concavity of $\log$,
\[
  \sum_{i=1}^m \frac{a_i}{A}\log\!\left(\frac{z_i}{a_i}\right)
  \le \log\!\left(\sum_{i=1}^m \frac{a_i}{A}\frac{z_i}{a_i}\right)
  =\log\!\left(\frac{1}{A}\sum_{i=1}^m z_i\right)
  =\log\!\left(\frac{S}{A}\right).
\]
Equality holds only when $z_i/a_i$ is constant in $i$, and the constraint $\sum_i z_i=S$ forces that constant to be $S/A$.
\end{proof}

\section{Optimization on a fixed support}
Let $A\subseteq\{1,\dots,n\}$ be nonempty and proper. Write
\[
  P_A:=\sum_{i\in A} p_i, \qquad Q_A:=\sum_{i\in A} q_i.
\]
We ask for the best strategy whose explicit bets are confined to $A$, so that $x_j=0$ for $j\notin A$.

\begin{proposition}[Fixed-support optimizer]\label{prop:fixed-support}
Let $A\subsetneq\{1,\dots,n\}$ be nonempty and assume $Q_A<1$. Define
\begin{equation}\label{eq:cash-A}
  c_A:=\frac{1-P_A}{1-Q_A}.
\end{equation}
If
\begin{equation}\label{eq:positivity-A}
  p_i>c_A q_i \qquad (i\in A),
\end{equation}
then among all strategies with $x_j=0$ for $j\notin A$, the unique maximizer of \eqref{eq:kelly-problem} is
\begin{equation}\label{eq:fixed-support-solution}
  c=c_A,\qquad x_i=p_i-c_Aq_i \ (i\in A),\qquad x_j=0 \ (j\notin A).
\end{equation}
Its terminal wealth is
\begin{equation}\label{eq:wealth-levels-fixed-support}
  W_i=\frac{p_i}{q_i}=r_i \quad (i\in A),
  \qquad
  W_j=c_A \quad (j\notin A).
\end{equation}
\end{proposition}

\begin{proof}
Fix $c\in(0,1)$. For $i\in A$, define the \emph{effective total stake}
\[
  y_i:=x_i+cq_i.
\]
This is natural because $cq_i$ staked on outcome $i$ would pay exactly $c$ in state $i$; cash is therefore an implicit position in every outcome. For $i\in A$,
\[
  W_i=c+\frac{x_i}{q_i}=\frac{y_i}{q_i},
\]
while for $j\notin A$ we have $W_j=c$.

Since $x_j=0$ for $j\notin A$ and $c+\sum_i x_i=1$,
\[
  \sum_{i\in A} y_i
  =\sum_{i\in A} x_i+cQ_A
  =1-c+cQ_A
  =1-c(1-Q_A).
\]
Therefore, for fixed $c$, maximizing expected log wealth is equivalent to maximizing
\[
  \sum_{i\in A} p_i\log y_i
\]
over $y_i>0$ with
\[
  \sum_{i\in A} y_i=1-c(1-Q_A),
\]
since the terms $-\sum_{i\in A} p_i\log q_i +(1-P_A)\log c$ are constant. By Lemma \ref{lem:full-investment}, the unique maximizer for this fixed $c$ is
\[
  y_i=\frac{1-c(1-Q_A)}{P_A}p_i \qquad (i\in A).
\]
Substituting back yields the value function
\[
  \Phi_A(c)=C_A+P_A\log\!\bigl(1-c(1-Q_A)\bigr)+(1-P_A)\log c,
\]
where $C_A$ is constant in $c$. Since $\Phi_A$ is strictly concave on $(0,1)$, its unique maximizer is characterized by
\[
  \Phi_A'(c)= -\frac{P_A(1-Q_A)}{1-c(1-Q_A)}+\frac{1-P_A}{c}=0,
\]
which gives \eqref{eq:cash-A}. At $c=c_A$ one has
\[
  1-c_A(1-Q_A)=P_A,
\]
so the optimal effective stakes are simply $y_i=p_i$, hence $x_i=p_i-c_Aq_i$. Condition \eqref{eq:positivity-A} guarantees that these explicit stakes are strictly positive on $A$. Formula \eqref{eq:wealth-levels-fixed-support} follows immediately.
\end{proof}

The proposition explains the formula: if cash level $c_A$ is held back, then the bettor already has an implicit stake $c_Aq_i$ in every outcome. On the active support, the total stake should therefore equal $p_i$, so the additional explicit stake is $p_i-c_Aq_i$.

If $A$ is proper but $Q_A\ge 1$, then the same value function $\Phi_A(c)$ is strictly increasing on $(0,1)$, so no optimizer confined to $A$ can keep every state in $A$ active; the optimum necessarily collapses to a smaller support. Thus any genuinely active proper support must satisfy $Q_A<1$.

\section{Greedy support selection}
Sort the outcomes so that
\[
  r_1\ge r_2\ge \cdots \ge r_n.
\]
For $k=0,1,\dots,n$, let
\[
  A_k:=\{1,\dots,k\},
  \qquad P_k:=\sum_{i=1}^k p_i,
  \qquad Q_k:=\sum_{i=1}^k q_i,
\]
with the convention $P_0=Q_0=0$ and $c_0=1$. Whenever $k<n$ and $Q_k<1$, define
\[
  c_k:=\frac{1-P_k}{1-Q_k}.
\]
For notational convenience set $c_n:=0$.

\begin{theorem}[Greedy characterization of the canonical Kelly strategy]\label{thm:greedy}
Start from $A_0=\varnothing$ and $c_0=1$. At step $k<n$, compare $r_{k+1}$ with $c_k$.
\begin{itemize}[leftmargin=2em]
  \item If $r_{k+1}>c_k$, add outcome $k+1$ to the support and update cash to
  \[
    c_{k+1}=\frac{1-P_{k+1}}{1-Q_{k+1}}.
  \]
  \item If $r_{k+1}\le c_k$, stop.
\end{itemize}
Let $k^*$ be the first index at which the algorithm stops, with the convention $k^*=n$ if no stop occurs earlier. Then a canonical optimal strategy is
\begin{equation}\label{eq:global-solution}
  c^*=c_{k^*},
  \qquad
  x_i^*=\bigl(p_i-c^*q_i\bigr)_+ \qquad (i=1,\dots,n),
\end{equation}
and the unique optimal terminal wealth vector is
\begin{equation}\label{eq:wealth-floor}
  W_i^*=\max\!\left\{c^*,\frac{p_i}{q_i}\right\}.
\end{equation}
In particular, the active outcomes form a prefix after sorting by $p_i/q_i$. If $r_i\ne c^*$ for every $i$, then the optimal strategy itself is unique.
\end{theorem}

\begin{proof}
Suppose first that $k<n$ and $r_{k+1}>c_k$. Then
\[
  q_{k+1}<\frac{p_{k+1}}{c_k}\le \frac{1-P_k}{c_k}=1-Q_k,
\]
so $Q_{k+1}<1$ and $c_{k+1}$ is well defined. Moreover,
\[
  c_{k+1}<c_k
  \iff \frac{1-P_{k+1}}{1-Q_{k+1}}<\frac{1-P_k}{1-Q_k}
  \iff p_{k+1}>c_kq_{k+1}
  \iff r_{k+1}>c_k.
\]
Thus every accepted addition strictly lowers cash. Since the $r_i$ are decreasing,
\[
  r_i\ge r_{k+1}>c_k>c_{k+1} \qquad (1\le i\le k+1),
\]
so Proposition \ref{prop:fixed-support} applies to $A_{k+1}$. In particular, once an outcome becomes active, it never drops out.

Now suppose the algorithm stops at $k=k^*<n$, so $r_{k+1}\le c_k$. Consider any strict enlargement $B=A_k\cup S$, where $S\subseteq\{k+1,\dots,n\}$ is nonempty. Write
\[
  P_S:=\sum_{j\in S} p_j,
  \qquad
  Q_S:=\sum_{j\in S} q_j.
\]
Because $r_j\le r_{k+1}\le c_k$ for every $j\in S$,
\begin{equation}\label{eq:PS-bound}
  P_S\le c_kQ_S.
\end{equation}
If $Q_B<1$, the cash level attached to support $B$ is
\[
  c_B:=\frac{1-P_k-P_S}{1-Q_k-Q_S},
\]
and a one-line calculation gives
\[
  c_B<c_k \iff P_S>c_kQ_S.
\]
By \eqref{eq:PS-bound} this cannot happen, so $c_B\ge c_k$. Hence every $j\in S$ satisfies
\[
  r_j\le c_k\le c_B,
\]
which contradicts the strict positivity condition $r_j>c_B$ required by Proposition \ref{prop:fixed-support} for an optimizer with support confined to $B$. Therefore no strict enlargement of $A_k$ with $Q_B<1$ can be optimal. If $Q_B\ge 1$, the preceding remark after Proposition \ref{prop:fixed-support} shows that an optimizer confined to $B$ cannot keep all states in $B$ active, so its actual active support is smaller and has already been ruled out.

Any smaller prefix $A_m$ with $m<k$ is also suboptimal. Indeed, the algorithm did not stop at $m$, so $r_{m+1}>c_m$. By the first part of the proof, Proposition \ref{prop:fixed-support} applies to $A_{m+1}$ and produces a distinct optimizer among strategies confined to $A_{m+1}$. Since the optimizer on $A_m$ is feasible for that larger class but not optimal there, $A_m$ cannot be globally optimal.

Thus the canonical optimizer is attained at the stopping prefix $A_k$. Formula \eqref{eq:global-solution} is exactly the fixed-support formula there, extended by zeros off the active set, and \eqref{eq:wealth-floor} is just a restatement of the active and inactive wealth levels.

If the algorithm never stops before $n$, then every prefix extension is accepted and the canonical choice is $c^*=0$, $x_i^*=p_i$ for all $i$, which again yields \eqref{eq:wealth-floor}. Finally, strict concavity in the terminal wealth vector gives uniqueness of $W^*$, and the strategy is unique whenever no ratio $r_i$ lies exactly at the cutoff $c^*$.
\end{proof}

\begin{remark}[Binary reduction]
For a binary event, if only outcome $1$ is active then
\[
  c^*=\frac{1-p_1}{1-q_1},
  \qquad
  x_1^*=p_1-c^*q_1=\frac{p_1-q_1}{1-q_1},
\]
which is the classical one-bet Kelly fraction written in state-price form.
\end{remark}

\begin{remark}[Interpretation]
The optimal terminal wealth has the clipped form
\[
  W_i^*=\max\{c^*,r_i\}.
\]
On outcomes with sufficiently large edge ratio $p_i/q_i$, the bettor raises wealth from the cash floor $c^*$ up to $p_i/q_i$; all other outcomes are left at the floor. The cutoff $c^*$ is therefore the unique level at which the set of favorable states is separated from the rest.
\end{remark}

\section{Concluding comment}
For a single multinomial event, the full Kelly solution is already known and completely explicit. What is useful in practice is a derivation that keeps the intuition visible. The implicit-position viewpoint does exactly that: cash is an all-state claim, so active bets only need to top up each favorable outcome from $cq_i$ to $p_i$. Once written this way, the threshold rule and greedy support selection become immediate.

\end{document}